\documentclass[11pt]{article}
\usepackage{amsmath,amssymb,latexsym}
\usepackage{amsthm} 
\usepackage{color}
\usepackage{amscd}
\usepackage{amsfonts}

\newcommand{\g}{{\mathfrak g}}
\newcommand{\h}{{\mathfrak h}}
\newcommand{\m}{{\mathfrak m}}

\newcommand{\s}{{\mathfrak s}}
\newcommand{\mt}{\mathfrak t}
\newcommand{\n}{{\mathfrak n}}

\newcommand{\p}{{\mathfrak p}}

\newcommand{\mk}{{\mathfrak k}}
\newcommand{\ma}{{\mathfrak a}}
\newcommand{\mb}{{\mathfrak b}}
\newcommand{\mc}{{\mathfrak c}}

\newcommand{\q}{{\mathfrak q}}
\newcommand{\0}{{\bf 0}}
\newcommand{\C}{{\bf C}}
\newcommand{\R}{{\bf R}}

\newcommand{\Z}{{\bf Z}}

\newcommand{\cL}{{\mathcal{L}}}

\theoremstyle{plain} 
\newtheorem{Th}{\indent\sc Theorem}
\newtheorem{Lm}{\indent\sc Lemma}
\newtheorem{Cr}{\indent\sc Corollary}
\newtheorem{Ps}{\indent\sc Proposition}

\theoremstyle{definition}
\newtheorem{Df}{\indent\sc Definition}
\newtheorem{Rm}{\indent\sc Remark}
\newtheorem{Ex}{\indent\sc Example}

\newcommand{\im}{\sqrt{-1}\,}

\def\address#1#2{\begingroup
\noindent\parbox[t]{7.8cm}{
\small{\scshape\ignorespaces#1}\par\vskip1ex
\noindent\small{\itshape E-mail address}
\/: #2\par\vskip4ex}\hfill
\endgroup}
\makeatother

\title{
\huge Compact Homogeneous Locally Conformally K\"ahler Manifolds
}

\author{
\textsc{Keizo Hasegawa and Yoshinobu Kamishima}
}
\date{}

\begin{document}

\maketitle


\begin{abstract}
In this paper we show as main results two structure theorems of a compact homogeneous
locally conformally K\"ahler (or shortly l.c.K.) manifold,
a holomorphic structure theorem asserting that it has a structure of holomorphic principal fiber bundle over
a flag manifold with fiber a 1-dimensional complex torus, and a metric structure theorem asserting that
it is necessarily of Vaisman type. We also discuss and determine l.c.K. reductive Lie groups and compact
locally homogeneous l.c.K. manifolds of reductive Lie groups.
\end{abstract}

\section*{Introduction} 

A {\em locally conformally K\"ahler structure} 
({\em l.c.K. structure} for short)
on a differentiable manifold $M$
is a Hermitian structure $h$ on $M$ with its associated
fundamental form $\Omega$ satisfying 
$d \Omega = \theta \wedge \Omega$ for some
closed $1$-form $\theta$ (which is so called Lee form).
A differentiable manifold $M$ is called a 
{\em locally conformal K\"ahler manifold} 
({\em l.c.K. manifold} for short) if $M$ admits a l.c.K. structure.
Note that l.c.K. structure $\Omega$ is globally conformally
K\"ahler (or K\"ahler)
if and only if $\theta$ is exact (or 0 respectively); and a compact l.c.K.
manifold of non-K\"ahler type (i.e. the Lee form is neither $0$ nor exact)
never admits a K\"ahler structure  (compatible with the complex structure).
\smallskip

There have been recently extensive studies on l.c.K. manifolds 
(c.f. \cite{V3} \cite{DO}, \cite{KO}, \cite{Be}, \cite{GO}).
In this paper we are concerned with l.c.K. structures on
homogeneous and locally homogeneous spaces of Lie groups.
There exist many examples of compact non-K\"ahler
l.c.K. manifolds which are homogeneous or locally homogeneous
spaces of certain Lie groups, such as Hopf surfaces,
Inoue surfaces, Kodaira surfaces, or some class of elliptic
surfaces (c.f. \cite{Be}, \cite{Has}).
Their l.c.K. structures
are {\em homogeneous} or {\em locally homogeneous} in the sense
we will explicitly define in this paper (Definitions 1 or 2 respectively).
Note that homogeneous l.c.K. structures on
Lie groups are nothing but left-invariant l.c.K. structures, which
can be considered as l.c.K. structures on their Lie algebras.
\smallskip

In this paper we show as main results two structure theorems of
a compact homogeneous l.c.K. manifold, a holomorphic structure theorem asserting that
it is a holomorphic principal fiber bundle over a flag manifold with
fiber a $1$-dimensional complex torus (Theorem 1),  and a metric structure theorem asserting that
it is of {\em Vaisman type}, that is, the Lee form is parallel with respect to the Hermitian metric (Theorem 2).
It should be noted that the same structure theorem was proved by Vaisman ({\cite{V2})
for compact homogeneous l.c.K. manifolds of Vaisman type.
As a simple application of the theorem,
we can show that only compact homogeneous l.c.K. manifolds
of complex dimension $2$ are Hopf surfaces of homogeneous type (Theorem 3), and that
there exist no compact complex homogeneous l.c.K. manifolds;
in particular, no complex paralellizable manifolds admit their compatible l.c.K. structures (Corollary 3).
\smallskip

We will take the following key strategies to prove the main theorems.  A compact homogeneous l.c.K.
manifold $M$ is expressed as $M=G/H$, where $G$ is a compact Lie group and $H$ is a closed subgroup
of $G$. Since the Lie algebra $\g$ of $G$ is a reductive, $\g$ can be written as $\g = \mt + \s$,
where $\mt$ is the center of $\g$ and $\s=[\g,\g]$ is a semi-simple ideal of $\g$.
Our first observation is (1) $\g$ must satisfies $1 \le {\rm dim}\, \mt \le 2$. As the second observation,
applying a result of Hochschild and Serre, (2) we can express a l.c.K. form $\Omega$ as
$\Omega=-\theta \wedge \psi+d \psi$, where $\theta$ is the Lee form and $\psi$ is a $1$-form. 
Let $\xi \in \g$ be the Lee field (the associated vector field to $\theta$ w.r.t. $h$). We put $\xi=t+s \,(t \in \mt, s \in \s)$.
We define the vector field $\eta=J \xi$ (Reeb field) for the complex structure $J$,
and the Reeb form $\phi$ (the associated $1$-form to $\eta$ w.r.t. $h$).
We will see as the third observation  (3) under the condition $\Omega$ is $Jt$-invariant, we have $\psi=\phi$ and 
$\g=\p + \mk$, where $\p=<t, \eta>=<t, Jt>=<\xi, \eta>$, and $\mk= {\rm ker}\, \theta \cap {\rm ker}\, \phi$.
In particular we can express $\Omega=-\theta \wedge \phi +d \phi$ with $\phi \in \wedge^2 \mk^*$.
As the fourth observation, since the closure $K$ of the $1$-parameter subgroup of $G$ generated by $Jt$
is compact, (4) we can use the averaging method to make $\Omega$ on $M$
invariant by ${\rm Ad}(K)$: $\bar{\Omega}= \int_{K} \,{\rm Ad}(x)^* \Omega$ while preserving the complex structure $J$.
\smallskip

Our fifth observation is (5) we can consider a compact homogeneous l.c.K. manifold $M$
up to holomorphic isometry as $M=G/H$ with a homogeneous l.c.K. structure $(\Omega, J)$, 
satisfying $\g= \mt+\s\; ({\rm dim}\,\mt=1)$; and up to biholomorphism, as such with
a $Jt$-invariant l.c.K. form $\bar{\Omega}$. In particular we can express $M=S^1 \times _\Gamma S/H_0$,
where $S$ is a simply connected semi-simple Lie group, $H_0$ is the connected component of $H$ and
$\Gamma$ is a finite abelian group. These observations lead to Theorem 1. As for the proof of Theorem 2,
we have the sixth observation (6) the Lee form $\theta$ and the Reeb field $\eta$ are stable under the
averaging by $K$. In order to show it we need the seventh observation (7) we have a compact subgroup
$S^1 \times N_S(H_0)/H_0$ imbedded in $G/H_0=S^1 \times S/H_0$ as a l.c.K. manifold. We also
need a  classification of l.c.K. compact Lie algebras. We will see as the eighth observation (8) a reductive Lie algebra
admits a l.c.K. structure if and only if  ${\rm dim}\, \mt = 1$ and ${\rm rank} \, \s =1$.
In particular a compact Lie algebra admits a homogeneous l.c.K. structure if and only if it is ${\mathfrak u}(2)$;
and any homogeneous l.c.K. structure on a compact Lie group is of Vaisman type (Theorem 4).

\section{Preliminaries} 

In this section we review some terminologies and basic results
in the field of homogeneous spaces and l.c.K geometry, relevant to our arguments 
on homogeneous and locally homogeneous l.c.K. structures in this paper.
\smallskip

\begin{Df}{\rm
{\em A homogeneous locally conformally K\"ahler} ({\em homogeneous l.c.K.} for short)
manifold $M$ is a homogeneous Hermitian manifold with
its homogeneous Hermitian structure $h$, defining a locally
conformally K\"ahler structure $\Omega$ on $M$.
}
\end{Df}

\begin{Df}{\rm
If a simply connected homogeneous l.c.K. manifold $M=G/H$,
where $G$ is a connected Lie group and $H$ a closed
subgroup of $G$, admits a free action of a discrete subgroup
$\Gamma$ of $G$ on the left, then we call a double coset
space $\Gamma \backslash G/H$ 
{\em a locally homogeneous l.c.K.} manifold.
}
\end{Df}

A homogeneous manifold $M$ can be written as $G/H$, where $G$ is
a connected Lie group with closed Lie subgroup $H$.
If we take the universal
covering Lie group $\widehat{G}$ of $G$ with the projection
$p: \widehat{G} \rightarrow G$ and the pull-back
$\widehat{H} = p^{-1}(H)$ of $H$, then we have the
universal covering $\widehat{M} =\widehat{G}/H_0$ of $M$,
where $H_0$ is the connected component of the identity of
$\widehat{H}$; and $\Gamma = \widehat{H}/H_0$  is the fundamental group
of $M$ acting on the right. In case $G$ is compact, $\widehat{G}$
is of the form $\R^k \times S \,(k \ge0)$, where $S$ is a simply connected compact semi-simple
Lie group. It is also known that $G$ has a finite normal covering $\tilde{G}$ of the form
$T^k \times S$ with the projection $\tilde{p}: \tilde{G} \rightarrow G$;
and a compact homogeneous manifold
$M=G/H$ can be expressed as $\tilde{G}/\tilde{H} = T^k \times_{\Gamma}  S/\tilde{H}_0$,
where $\tilde{H}_0$ is the connected component of the identity of $\tilde{H}=\tilde{p}^{-1} H$ and
$\Gamma= \tilde{H}/\tilde{H}_0$ is a finite group acting on 
$\tilde{M}=T^k \times  S/\tilde{H}_0$ on the right. 
\smallskip

In  case $M$ is a homogeneous
l.c.K. manifold, $\widehat{M}$ is also a homogeneous l.c.K. manifold;
and since the Lee form $\widehat{\theta}= p^{-1} {\theta}$ is exact
the fundamental form $\widehat{\Omega}= p^{-1}{\Omega}$ is globally
conformal to a K\"ahler structure ${\omega}$. The Lie group $\widehat{G}$
acts holomorphically and homothetically on $(\widehat{M}, {\omega})$ on the left;
and the fundamental group $\Gamma$
acts likewise on $(\widehat{M}, {\omega})$ on the right. 
Conversely, a K\"ahler structure $\omega$ on $\widehat{M} =\widehat{G}/H_0$
with holomorphic and homothetic action of $\widehat{G}$
on the left and $\Gamma$ on the right defines a homogeneous l.c.K. structure
$\Omega$ on $M=G/H$, where $H=H_0 \rtimes \Gamma$ with 
$\Gamma \cap H_0 = \{0\}$ and $\Gamma \subset N_{\widehat{G}}(H_0)$.
If $\Gamma$ is a discrete subgroup of $\widehat{G}$ acting properly
discontinuously and freely on $\widehat{G}/H_0$ on the left, then we can define
a locally homogeneous l.c.K. structure on $\Gamma \backslash \widehat{G}/H_0$.
In particular, for a simply connected Lie group $G$ with
a left invariant l.c.K. structure $\Omega$ and a discrete subgroup $\Gamma$ of $G$,
$\Omega$ induces a locally homogeneous l.c.K. structure $\widetilde{\Omega}$ on
$\Gamma \backslash G$.
\smallskip

Let $M=G/H$ be a homogeneous space of a connected Lie group $G$ with closed
subgroup $H$. Then the tangent space of $M$ is given as a $G$-bundle 
$G \times_H \g/\h$ over $M=G/H$ with fiber $\g/\h$, where the action of $H$ on the
fiber is given by ${\rm Ad}(x) \, (x \in H)$. A vector field
on $M$ is a section of this bundle; and a $p$-form on $M$ is a section of
$G$-bundle $G \times_H \wedge^p (\g/\h)^*$, where the action of $H$ on
the fiber is given by ${\rm Ad}(x)^* \, (x \in H)$. An invariant vector field (respectively
$p$-form), the one which is invariant by the left action of $G$, is canonically identified
with an element of $(\g/\h)^H$ (respectively $(\wedge^p (\g/\h)^*)^H$),
which is the set of elements of $\g/\h$ (respectively $\wedge^p (\g/\h)^*$) invariant by
the adjoint action of $H$. A complex structure $J$ on $M$ is likewise considered
as an element $J$ of ${\rm Aut}(\g/\h)$ such that $J^2=-1$ and
${\rm Ad}(x) J=J {\rm Ad}(x) \, (x \in H)$. Note that we may also consider an invariant
$p$-form as an element of $\wedge^p \g^*$ vanishing on
$\h$ and invariant by the action ${\rm Ad}(x)^* \, (x \in H)$. 
\smallskip

We recall that $\g$ is {\em decomposable} with respect to $H$
if there is a direct sum decomposition of $\g$ as
$$\g = \m + \h,$$
for a subspace $\m$ of $\g$ and
${\rm Ad}(x)(\m) \subset \m$ for any $x \in H$. 
This is the case, for instance, when $H$ is a reductive Lie group.
In case $\g$ is decomposable, the tangent space of $M=G/H$ is given by
the $G$-bundle $G \times_H \m$ over $M=G/H$, identifying $\g/\h$ with $\m$.
An invariant vector field (respectively $p$-form) on $M$ is
identified with an element of $\m^H$ (respectively $(\wedge^p (\m)^*)^H$),
which is the set of elements of $\m$ (respectively $\wedge^p (\m)^*$) invariant by
the adjoint action of $H$. 
A complex structure $J$ on $M$ can be considered as
an element $J$ of ${\rm Aut}(\m)$ such that $J^2 =-1$ on
$\m$ and ${\rm Ad}(x) J=J {\rm Ad}(x) \, (x \in H)$. It is also convenient to
consider a complex structure $J$ on $M$ as
an element $J$ of ${\rm End}(\g)$ such that $J^2 =-1$ on
$\m$, $J \h=0$ and ${\rm Ad}(x) J=J {\rm Ad}(x) \, (x \in H)$ (c.f. \cite{K}).
\smallskip

An invariant vector field $X \in \m^H$ generates a global $1$-parameter group of 
diffeomorphisms on $M=G/H$ given by the right action of ${\exp}\, t X$:
$$\phi : \R \times G/H \longrightarrow G/H, \,\phi (t, gH)= g ({\exp}\, t X) H.$$
Since the closure $K$ of the $1$-parameter subgroup of $G$ generated by $X$
is compact, we can use the averaging method to make differential forms $\omega$ on $M$
invariant by ${\rm Ad}(K)$:
$$ \int_{K} \,{\rm Ad}(x)^* \omega.$$
For a l.c.K. form $\Omega$ with its Lee from $\theta$, we can average $\Omega$ to
make a ${\rm Ad}(K)$-invariant l.c.K. form $\overline{\Omega}$ under the condition that
the action is compatible with the complex structure $J$.
Note that we have the Lee form $\overline{\theta}$ identical with $\theta$,
but since the metric $\overline{h}$ is in general different from $h$
its associated Lee field $\overline{\xi}$ is in general different from $\xi$.
\smallskip

For a $\g$-module $M$, we can define $p$-cochains as the $p$-linear
alternating functions on $\g^p$, which are $\g$-modules defined by
$$(\gamma f) (x_1, x_2, ...,x_p)= \gamma f(x_1, x_2, ...,x_p) - 
\sum_{i=1}^{p} f(x_1, ...,x_{i-1}, [\gamma, x_i], x_{i+1}, ...,x_p),$$
where $\gamma \in \g$ and $f$ is a $p$-cochain (cf. \cite{HS}).
The coboundary operator is defined by
\begin{eqnarray*}
(d f) (x_0, x_1, ..., x_p) & = & \sum_{i=0}^{p} (-1)^i x_i f(x_0, ..., \widehat{x_i}, ...,x_p) \\
& + & \sum_{j<k} (-1)^{j+k} f([x_j, x_k], x_0, ..., \widehat{x_j}, ..., \widehat{x_k}, ..., x_p).
\end{eqnarray*}

We are interested in the case when a $\g$-module is defined by
the representation of $\g$ on $\R$, assigning $X \in \g$ to $-\theta(X)$ for
the Lee form $\theta$ on a l.c.K. Lie algebra $\g$.
The corresponding coboundary operator is given by
$$d_\theta: w \rightarrow -\theta \wedge w + d w,$$
and its cohomology group $H_{\theta}^p(\g, \R)$ is called the {\em $p$-th twisted cohomology
group} with respect to the Lee form $\theta$. The condition of l.c.K. structure $\Omega$ on
$\g$ is expressed by $d_\theta \Omega=0$.
We know (\cite{HS}) that for a reductive Lie algebra $\g$, 
all of the cohomology groups $H_{\theta}^p(\g, \R) \; (p \ge 0)$ vanish; and in particular
we have $\Omega= -\theta \wedge \psi + d \psi$ for some $1$-form $\psi$.

\section{A holomorphic structure theorem of compact homogeneous l.c.K. manifolds} 

In this section we prove a structure theorem of compact
homogeneous l.c.K. manifolds, which asserts that such a
compact complex manifold is biholomorphic to a holomorphic
principal bundle over a flag manifold with fiber a $1$-dimensional
complex torus. This result may be compared
with the well-known theorem (due to Matsushima \cite{M}) that a compact homogeneous K\"ahler
manifold is biholomorphic to a K\"ahlerian product of
a complex torus and a flag manifold.
\smallskip

Let $M$ be a compact homogeneous l.c.K. manifold of dimension $(2m+~2),$ $m \ge 1$,
with its associated fundamental form $\Omega$ and Lee form $\theta$,
satisfying $d \Omega = \theta \wedge \Omega$. $M$ can be written as $G/H$, where $G$
is a connected holomorphic isometry group 
of the Hermitian manifold $(M, h)$ and $H$ a compact subgroup of $G$
which contains no normal Lie subgroups of $G$.
Since $G$ is a closed subgroup of the isometry group of $(M, h)$,
it is a compact Lie group; in particular $G$ is {\em reductive},
that is, the Lie algebra $\g$ of G can be written as
$$\g = \mt + \s$$
\noindent where $\mt$ is the center of $\g$ and  $\s$ is
a semi-simple Lie algebra.
Let $\h$ be the Lie algebra of $H$. Then $\g$ also admits a
decomposition: 
$$\g = \m + \h$$
\noindent satisfying ${\rm Ad}(x)(\m) \subset \m \,(x \in H)$
for a subspace $\m$ of $\g$. Note that we have also $\mt \cap \h = {0}$.
Since the Lee form $\theta$ is invariant, its
associated vector field $\xi$ (which is called {\em Lee field}) with respect
to the metric $h$ is also invariant; and thus $\xi$ may be taken as an element of $\m$
invariant by ${\rm Ad}(x)$ for any $x \in H$.
\smallskip

Any invariant form on $M$ can be considered as an element 
of $\wedge^p \g^*$ vanishing on
$\h$ and invariant by the action ${\rm Ad}(x)^* \, (x \in H)$.
In particular, we consider $\Omega, \theta$
as the elements of $\wedge \g^*$ satisfying these conditions and
$$d \Omega = \theta \wedge \Omega\,.$$

From now on we assume $M$ is of non-K\"ahler type; 
and thus $\theta$ is a non-zero, closed but not exact form on $\g$.
Note that since $\s = [\g, \g]$ and $\theta$ is a non-zero closed form,
$\theta([X, Y])=d \theta (X, Y) =0$ for all
$X, Y \in \g$ and thus $\theta$ vanishes on $\s$. In particular
we must have ${\rm dim}\,\mt \ge 1$ and $\theta \in \mt^*$.
\smallskip

The Lee field $\xi \in \m$ may be expressed as $\xi=t+s, \, t \in \mt \,(t \not=0), s \in \s$,
where $\xi$ is normalized, satisfying $h(\xi,\xi)=1$ and thus $\theta(\xi)=\theta(t)=1$.
We define the Reeb field $\eta \in \m$ as $\eta=J \xi$
with its associated $1$-form $\phi$ satisfying $\phi(\eta)=1$.
We can express $\g$ as
$$\g = <\xi, \eta> +\,\mk,$$
where $<\xi,\eta>$ is the $2$-dimensional subspace of $\g$
generated by $\xi$ and $\eta$ over $\R$, and
$\mk={\rm ker}\,\theta \cap {\rm ker}\,\phi$ with $\mk \supset \h$.
Note that $h(\xi, \eta)=\Omega(\eta,\eta)=0$ and $<\xi,\eta>$ is orthogonal to $\mk$
with respect to $h$.
\smallskip

It is known (due to Hochschild and Serre \cite{HS}) that there exists a 1-form $\psi \in \g^*$
such that 
$$\Omega= -\theta \wedge \psi+ d \psi,$$
where $\psi$ defines an invariant $1$-form on $M$: $\psi$ vanishes on $\h$
since we have $\psi(\h)=\Omega(\h, t)=0$; and
$\psi$ is ${\rm Ad}(x)$-invariant for $x \in H$ since we have
$\psi([\h, Y])= d \psi(\h, Y)=\Omega(\h, Y)=0$.
We set $\psi_c = \psi - c \,\theta$ for $c \in \R$. Note that we have $d \psi_c = d \psi$; and
$$\Omega= -\theta \wedge \psi_c+ d \psi_c.$$

\begin{Lm} 
There exists $\sigma \in \g$ and $c \in \R$
such that 
$$\psi_c(\sigma)=1,\, \psi_c(t)=0,\, \theta(t)=1,\, \theta(\sigma)=0,$$
and  $d \psi_c(\sigma,Y)=0$ for all $Y \in \g$. 
\end{Lm}

\begin{proof}
 Since $\theta$ and $\psi$ are linearly independent, 
we can take an element $\sigma'$ such that $\psi(\sigma')=1$ and $\theta(\sigma')=0$.
If $\psi(t) \not=0$, then take $\psi_c = \psi- c \,\theta$ for $c=\psi(t)$ satisfying
$\psi_c(t)=0$. Then we have  $\psi_c(\sigma')=1,\,\theta(t)=1, \,\psi_c(t)=\theta(\sigma')=0.$
Note that since $d \psi_c(t, \sigma')= \psi_c([t,\sigma'])=0$, we have $\Omega(\sigma',t)=1$;
in particular $\sigma' \notin \h$.
\smallskip

Recall that for a bilinear form $\Phi$ on a vector space $V$,
 $${\rm Rad}\, \Phi = \{u \in V\, |\, \Phi(u, v)= 0 \; {\rm for}\; {\rm any}\; v \in V\}.$$
Let $\p'=<t, \sigma'>$ and $\q = {\rm Ker}\, \theta \cap {\rm Ker}\, \psi_c={\rm Ker}\, \theta \cap {\rm Ker}\, \psi$
with $\q \supset \h$.
Then we have an orthogonal direct sum with respect to $\Omega$:
$$\g = \p' + \q, \; \p' \cap \q = \{0\}.$$

We first note that $\Omega |\q = d \psi_{c}$ is non-degenerate on $\q \, ({\rm mod}\; \h)$. 
In fact, suppose that there exists a non-zero element $v \in  \q$ such that
$d \psi_{c} (\q, v) = 0$.
Then for $v'=at+bv$ with some $a, b \in \R, b\not=0$, we have
$$\Omega(\sigma', v')= -(\theta \wedge \psi_{c}) (\sigma',v')+ d \psi_{c} (\sigma', v')=
a +b \,d \psi_{c}(\sigma',v) =0.$$
Since we also have $\Omega(t, v')=0$ and $\Omega(\q,v')=0$,
we have $\Omega(\g,v')=0$, 
contradicting the non-degeneracy of $\Omega$ on $\g \, ({\rm mod}\; \h)$.
\smallskip

Let $\chi$ be a $1$-form
defined on $\q$ by  $\chi(X)= d \,\psi_c(\sigma', X)$. Since $d \,\psi_c$ is
non-degenerate on $\q$, there exists $\tau \in \q$ such that $\chi(X)=d \,\psi_c(\tau, X)$;
and thus $d \,\psi_c(\sigma'-\tau, X)=0$ for all $X \in \q$. Let $\sigma=\sigma'-\tau$ and
$\p=<t, \sigma>$, then we have an orthogonal direct sum with respect to $\Omega$:
$$\g = \p + \q, \; \p \cap \q = \{0\}.$$
and $\psi_c(\sigma)=1,\, \theta(\sigma)=0 \,(\sigma \notin \h)$. 
Since $d \psi_c(\sigma, t)= \psi_c([\sigma, t])=0$, we have 
$${\rm Rad}\, d \psi_c = \p  \, ({\rm mod}\; \h).$$ 
This completes the proof of Lemma 1.  
\end{proof}
\medskip

From now on we write $\psi_c$ simply as $\psi$.
\medskip

\begin{Cr} 
We have $J\xi=\sigma \, ({\rm mod}\; \h)$; and thus $\eta =\sigma \, ({\rm mod}\; \h)$.
\end{Cr}
\begin{proof}
By the definition, the Lee field $\xi$ satisfies that 
$h(\xi,X)=\theta(X)$; and thus $\Omega(J\xi,X)=\theta(X)$.
By Lemma 1 we have $\g = \p + \q$
where $\p=< t,\sigma>$ and 
$\q={\rm Ker}\, \theta \cap {\rm Ker}\, \psi$. 
Hence we have
$\Omega(J\xi,X)=0$ for all $X \in \q$,
$\Omega(J\xi,t)=1$ and $\Omega(J\xi,\sigma)=0$.
On the other hand,
since we have $\Omega=\psi \wedge \theta+d\psi$, we get
$\Omega(\sigma,X)=\psi(\sigma)\theta(X)-\psi(X)\theta(\sigma)+d\psi(\sigma,X)=
0$ for all $X \in \q$, and $\Omega(\sigma,t)=1$.
Hence we have  $J\xi=\sigma \, ({\rm mod}\; \h)$; and thus $\eta =\sigma \, ({\rm mod}\; \h)$,
where $\eta=J\xi$ is the Reeb field by definition. 
\end{proof}
\medskip

\begin{Cr} 
We have $\cL_\sigma\Omega=0$. 
\end{Cr}
\begin{proof}
We write $\Omega=\theta\wedge \psi+d\psi$.
Since we have  $\psi(\sigma)=1$ and $d\psi(\sigma,X)=0$ for all $X \in \g$,
we get 
$\cL_\sigma\psi=d\iota_\sigma\psi+\iota_\sigma d\psi=0.$
Since we have $\cL_\sigma(\theta\wedge \psi)=(\cL_\sigma\theta)\wedge \psi
-\theta\wedge\cL_\sigma\psi=(\cL_\sigma\theta)\wedge \psi$ and
$\cL_\sigma\theta=d\iota_\sigma\theta+\iota_\sigma d\theta=0$,
we get $\cL_\sigma\Omega=0.$
\end{proof}
\medskip

\begin{Cr} 
We have
$1 \le {\rm dim}\,\mt \le 2, \, \mt \subset \;<t, \sigma>  +\,\h.$
\end{Cr}

\begin{proof}
We have seen in Lemma 1 that $d \psi$ is non-degenerate on $\q \, ({\rm mod}\, \h)$. 
For any $X \in \mt$ written as $X=a t + b \sigma + Z \, (a, b \in \R, Z\in \q)$ and any $Y \in \q$,
we have $d \phi(Z,Y)=\Omega(Z,Y)=\Omega(X,Y)=0$; and thus $Z \in \h$. In particular,
we have  $\mt \cap {\q} = \mt \cap {\mathfrak h}= \{0\}$. Since ${\rm dim}\, \q = n-2$,
we must have $1 \le {\rm dim}\, \mt \le 2$.
\end{proof}
\medskip

\begin{Lm} 
Suppose that the l.c.K. form $\Omega$ is $J t$-invariant. Then,
$\p ={\rm Rad}\, d {\psi}$ as in Lemma~1 is generated by
$\{t, Jt\}$ or $\{\xi, \sigma\}$:
$$\p=<t, \sigma>=<t, Jt>=<\xi, \sigma>.$$

\end{Lm}

\begin{proof}
Let $\q'$ be the orthogonal complement of $<t, Jt>$ with respect to $\Omega$.
We show first that 
$\q'=\q = {\rm Ker}\, \theta \cap {\rm Ker}\, \psi$; and thus $\p=<t,Jt>$.
For $X \in \q'$, we have
$$d \Omega(X, Jt,t)=\theta(X) \Omega(Jt,t)=\theta(X) h(t,t).$$
On the other hand, we have
$$d \Omega(X, Jt,t)= \Omega([X, Jt],t)= -\Omega(X, [Jt,t])=0,$$
due to the invariance of $\Omega$ by ${\rm Ad}({\exp} Jt)$. 
Hence we have $X \in {\rm ker}\, \theta$.
For $X \in \q'$, we also have $\Omega(X,t)=\psi(X)=0$;
and thus $X \in {\rm ker}\, \psi$. Since $\q' \subset \q$ and ${\rm dim}\, \q'= {\rm dim}\, \q$,
we must have $\q'=\q$. Note that since $\p$ is $J$-invariant $\q$ is also the orthogonal
complement with respect to $h$.
\smallskip

We show that $\xi=t+b \sigma$ for $b \in \R$; and thus $\p=<\xi, \sigma>$.  We have
$$h(\xi, X)=\theta(X)=\Omega(\sigma,X)=0$$
for $X \in \q$; and thus $\xi \in \p$. If we write $\xi=a t+ b \sigma$, then $a=\theta(\xi)=1$.
\end{proof}
\medskip

\begin{Lm} 
If $\Omega$ is $J t$-invariant,
we have $\Omega= - \theta \wedge \phi + d \phi\,,\, d \phi \in \wedge^2 \, \mk^*$.
\end{Lm}

\begin{proof}
We have shown that $\p$ is generated
by $\{\xi, \sigma\}$; and $\q$ is the orthogonal complement of $\p$
with respect to both $\Omega$ and $h$.
Since $d \psi$ is non-degenerate on $\q \,({\rm mod}\, \h)$,
there exist $X_i, Y_j \in \q$, $i, j=1, 2,...,k \,(k \le m)$ which are linearly independent
and $d \psi = \sum \rho_i \wedge \tau_i$,
where $\rho_i, \tau_i$ are the dual forms corresponding to $X_i, Y_i$.
Since $\sigma \in {\rm Rad}\, d \psi$, we have
$$\Omega(X, \sigma)= -(\theta \wedge \psi)(X, \sigma) = -\theta(X)$$
for any $X\in \g$. Hence we have 
$$\Omega(J \sigma, \sigma)=-\theta(J \sigma) = -h(\xi, J \sigma) = -\Omega(\xi, \sigma)=1.$$
Since $h(\xi, \xi) = \Omega(J\xi, \xi)=1$, we can see $J\xi = \sigma$.
In fact, we can set $J \xi = \sigma + Z$ and $J \sigma = -\xi + Z'$ for
$Z \in \, <\xi, X_i, Y_j>, Z' \in \, <\sigma, X_i, Y_j>, i, j=1, 2,..., k$; and thus
we have $Z' = -J Z$. Then we have 
$$\Omega(\xi, J\xi)=\Omega(\sigma+Z, J\sigma+JZ) =\Omega(\sigma, J\sigma) +\Omega(Z, JZ),$$
$$\Omega(\sigma, J\sigma)=\Omega(-\xi+Z', -J\xi+JZ') =\Omega(\xi, J\xi) +\Omega(Z', JZ'),$$
from which we get $h(Z,Z)+h(Z',Z')=0$; and thus $Z=Z'=0$. Since $\eta=J \xi$
by definition we must have $\sigma=\eta$;  and thus $\q=\mk$ and $\psi=\phi$.
We can also see that $JX_i = Y_i$, $i, j=1,2,...,k$.
\end{proof}
\medskip

We have seen, under the assumption that $\Omega$ is $J t$-invariant,
that $\xi$ can be written as $\xi=t + b \eta$.
We have $\mt=<\xi, \eta>  \, ({\rm mod}\; \h)$ for the case ${\rm dim}\, \mt = 2$.
For the case ${\rm dim}\,\mt = 1$, we have $\g = \mt + \s$
with $\s=<\eta>+\mk$, and $t$ is a generator of $\mt$.
Note that the complex structure $J$ may be expressed
with respect to a basis
$\{t, \eta\}$ as $J t=b t+(1+b^2) \eta, J \eta=-t-b \eta$; and
$\theta = t^*, \phi = \eta^*-b t^* \; (t^*, \eta^* \in \g^*)$.
\medskip

\begin{Lm} 
Under the condition that $\Omega$ is $J t$-invariant,
we can reduce the case ${\rm dim}\,\mt = 2$ to the case
${\rm dim}\,\mt =1$. 
\end{Lm}

\begin{proof} 
First note that we have $\s=[\g,\g]=[\mk, \mk]$.
Since $d \phi \in \wedge^2 \, \mk^*$, and $d \phi = \sum \rho_i \wedge \tau_i$,
where $\rho_i, \tau_i$ are the dual forms corresponding to $X_i, Y_i \in \mk$, we have
$\eta \in [\mk,\mk]=\s \,({\rm mod}\, \h)$ with $\eta \not\in \h$. 
In case ${\rm dim}\, \mt = 2$, since we have $\mt=<\xi, \eta>  \, ({\rm mod}\; \h)$,
$\eta=s_1+h_1=t_2+h_2$ with $h_1, h_2 \in \h, s_1 \in \s, t_2 \in \mt$. Let $\g'$ be
the subalgebra of $\g$ generated by $\xi$ and $\s$, and $G'$
the Lie subgroup of $G$  corresponding to $\g'$ of $\g$. Then since we have
$\eta \in \s \,({\rm mod}\, \h)$, $G'$ acts on $M$ transitively; and $M$ can be written as
$G'/H'$ with its isotropy subgroup $H'=H \cap G'$. It is clear that the center $\mt'$ of $\g'$ is
generated by $t$, and thus ${\rm dim}\, \mt'=1$. The canonical injection 
$G' \hookrightarrow G$ induces a holomorphic isometry from $G'/H'$ to $G/H$.
\end{proof}
\medskip

Since $Jt$ is an invariant vector field compatible with $J$, 
satisfying ${\rm ad}(Jt) J= J {\rm ad}(Jt)$,
we can apply the averaging method to make a l.c.K. form $\overline{\Omega}$ invariant by
${\rm Ad}({\exp} Jt)$; in particular, we have
$$\overline{\Omega}([Jt, X],Y)+\overline{\Omega}(X,[Jt,Y])=0$$
for all $X,Y \in \g$, where $\overline{\Omega}$ defines a l.c.K structure on $M$
compatible with the original complex structure $J$. By Lemma 5 we can express $M=G'/H'$
with $\g'=\mt' + \s \;( {\rm dim}\,\mt' =1)$. Since $G'$ is a subgroup of $G$,
$G'$ preserves the original l.c.K. structure $(\Omega, J)$ on $M$ as well as the averaged
l.c.K. structure $(\bar{\Omega},J)$ on $M$. Therefore, we have the following key observation:

\begin{Rm} 
We may consider a compact homogeneous l.c.K. manifold $M$
up to holomorphic isometry as $M=G/H$ with a homogeneous l.c.K. structure $(\Omega, J)$, 
satisfying $\g= \mt+\s\; ({\rm dim}\,\mt=1)$; and up to biholomorphism, as such with
a $Jt$-invariant l.c.K. form $\bar{\Omega}$.
\end{Rm}
\medskip

\begin{Ps} 
A compact homogeneous l.c.K. manifold $M$ admits a holomorphic
flow, which is a Lie group homomorphism from
$\C^1$ to the holomorphic automorphism group of $M$. 
\end{Ps}

\begin{proof}
Let ${\rm Aut}(M)$ be the holomorphic automorphism group of $M$.
Then we know that ${\rm Aut}(M)$ 
is a complex Lie group with its associated complex Lie algebra
${\mathfrak a}(M)$ consisting of holomorphic vector fields on $M$.
Let ${\rm Isom} (M)$ be the (maximal connected) isometry group of $M$. 
Then we know that ${\rm Isom} (M)$ is a compact
real Lie group with its associated Lie algebras ${\mathfrak i}(M)$
consisting of all Killing vector fields on $M$.
Note that $G$ can be taken as the intersection of ${\rm Aut}(M)$ and  ${\rm Isom} (M)$
being a compact subgroup of ${\rm Isom} (M)$,
\smallskip 

Since $\xi \in <t, Jt>$ by Lemma 2, the Lee field $\xi$ is an infinitesimal
automorphism on $M$; and thus $\xi - \sqrt{-1}J \xi$ is a holomorphic vector field on $M$. 
Hence the homomorphism $\overline{\phi}$
of Lie algebras mapping $\xi - \sqrt{-1}J \xi$ to ${\mathfrak a}(M)$ induces
a homomorphism $\phi$ of complex Lie groups mapping $\C$
to ${\rm Aut}(M)$.
\end{proof}
\medskip

\begin{Th} 
A compact homogeneous l.c.K. manifold $M$ is, up to biholomorphism,
isomorphic to a holomorphic principal fiber bundle over a flag
manifold with fiber a $1$-dimensional complex torus $T_\C^1$.
\smallskip

To be more precise, $M$ can be written as a homogeneous space
form $G/H$, where $G$ is a 
compact connected Lie group of holomorphic automorphisms on $M$
which is of the form
$$G = S^1 \times S,$$
where $S$ is a compact simply connected semi-simple Lie group, 
including the connected component $H_0$ of $H$ which is a closed subgroup of $S$. 
$S/H_0$ is a compact simply connected homogeneous Sasaki manifold, 
which is a principal fiber bundle over a flag manifold $S/Q$
with fiber $S^1=Q/H_0$ for some parabolic subgroup $Q$ of $S$
including $H_0$. $M=G/H$ can be expressed as
$$M=S^1 \times_{\Gamma} S/H_0,$$
where $\Gamma = H/H_0$ is a finite abelian group
acting holomorphically on the fiber  $T_{\C}^1$ of the fibration $G/H_0 \rightarrow G/Q$ on the right.
\end{Th}

\begin{proof}
We can assume that $\g=\mt+\s$ with ${\rm dim}\, \mt =1$; and $\eta \in \s$.
Let ${\mathfrak q}=<\eta>+\h$, then since $[\eta, \h] \subset \h$,
${\mathfrak q}$ is a Lie subalgebra of $\s$; in fact we have
${\mathfrak q}=\{X \in \s \,|\, d \phi(X, \s)=0\}$.
Let $S$ and $Q$ be the corresponding
Lie subgroup of $G$, then $Q$ is a closed subgroup of $S$ since
we have $Q=\{ x \in S \,|\, {\rm Ad}(x)^* \phi=\phi \}$, 
which is clearly a closed
subset of $S$; in particular, $H_0$ is a normal subgroup of
$Q$ with $Q/H_0 =S^1$, and $\eta$ generates an $S^1$ action on $S$.
(cf. \cite{BW}). 
We have seen in Lemma 3 that $d \phi$ defines a homogeneous symplectic structure on
$S/Q$ compatible with the complex structure $J$, 
which is a K\"ahler structure on $S/Q$ (due to Borel \cite{Bo});
in particular $Q$ is a parabolic subgroup of $S$.
\smallskip

We have seen that the abelian Lie subalgebra $<\xi, \eta>=<t, \eta>$ of $\g$
generates a $2$-dimensional torus $T_{\R}^2$ action on $M$ where $t$ is a generator
of the center of $\g$ generating a $S^1$ action on $M$; and $\xi - \sqrt{-1} \eta$
generates a holomorphic $1$-dimensional complex torus action on $M=G/H$ on the right.
 We have $M=S^1 \times_{\Gamma} S/H_0$,
where $S/H_0 \rightarrow S/Q$ is a principal $S^1$-bundle over the flag manifold $S/Q$; 
and $\hat{M}=S^1 \times S/H_0 \rightarrow S/Q$ is a holomorphic principal fiber
bundle over the flag manifold $S/Q$ with fiber $T_{\C}^1$.
Since $H \subset Q$ and thus the holomorphic action of $\Gamma=H/H_0$ is trivial on
the base space $S/Q$, it actually acts on the fiber $T_{\C}^1$, inducing
a holomorphic principal fiber bundle $M \rightarrow S/Q$ with fiber $T_{\C}^1$.
\end{proof}
\medskip

\begin{Cr} 
There exist no compact complex homogeneous l.c.K. manifolds;
in particular, no complex paralellizable
manifolds admit their compatible l.c.K. structures.
\end{Cr}

\begin{proof}
We know that only compact complex Lie groups are complex tori, which can not
act transitively on compact l.c.K. manifolds.
\end{proof}

\section{A metric structure theorem of compact homogeneous l.c.K. manifolds} 

\begin{Df}{\rm
A l.c.K. manifold  $(M, h)$ is of {\em Vaisman type} if
the Lee field $\xi$ is parallel with respect to the Riemannian connection for $h$.
}
\end{Df}

For a homogeneous l.c.K. manifold $M=G/H$, 
the Lee field $\xi$ is parallel with respect to
the Riemannian connection for $h$ if and only if
$$h(\triangledown_X \xi, Y) = h([X, \xi],Y) - h([\xi,Y],X) + h([Y,X],\xi) =0$$
for all $X,Y \in \g$. Since the Lee form is closed: $h([Y,X],\xi) =0$,
this condition is equivalent to
$$h([\xi, X],Y) + h(X, [\xi,Y])= 0$$
for all $X,Y \in \g$. And this is exactly the case when the Lee field $\xi$ is Killing field. 
It should be also noted that $\xi$ is Killing if and only if $\cL_\xi \Omega=0$ and $\cL_\xi J=0$
for the l.c.K. form $\Omega$ and its compatible complex structure $J$.
\medskip

Let $\sigma$ be an element of $\g$ obtained in Lemma 1 for the original l.c.K. form $\Omega$.
We have the following key Lemma.

\begin{Lm} 
We have $\cL_\sigma J=0$.
\end{Lm}

\begin{proof}
We have seen (in Remark 1 and Theorem 1) that $M=G/H$ can be expressed as
$M=S^1 \times_{\Gamma} S/H_0$ with the original l.c.K. form $\Omega$,
where $\Gamma = H/H_0$ is a finite abelian group.
We have a compact Lie group $S^1\times N_S(H_0)/H_0$
imbedded in $\tilde{M}=S^1 \times S/H_0$; and a l.c.K. structure $(\hat\Omega, \hat{J})$ on 
$S^1\times N_S(H_0)/H_0$ can be induced from the l.c.K. structure $(\tilde{\Omega}, \tilde{J})$ on
$\tilde{M}$ by restriction, where $N_S(H_0)$ denotes
the normalizer of $H_0$ in $S$. 
In fact we can define a l.c.K. form $\hat\Omega$ just as the restriction on $\mt + \n_\s(\h)$ of 
the l.c.K. form $\Omega$ on $\g$; and since we have ${\rm ad}(X) J=J {\rm ad}(X) \, (X \in \h)$ 
with $J\h=0$, we can also define a complex structure $\hat{J}$ on $\mt + \n_\s(\h)/\h$ as 
the restriction of $J$ on $\mt+\n_\s(\h)$.
Note that we have $\mt =<t>$ and $\sigma, Jt \in \n_\s(\h)$.
\smallskip

For the case $\n_\s(\h) \supsetneq \q$, since $\mt + \n_\s(\h)/\h$ is a compact l.c.K. Lie algebra
it must be ${\mathfrak u}(2)=\R \oplus {\mathfrak su}(2)$ by Theorem 4; in particular $\hat{\Omega}$ is $Jt$-invariant. 
Applying Lemma 2 we have $\sigma\in < t,Jt>$. Since $\cL_{Jt}J=0$ and $\cL_Y J=0$ for all $Y\in \h$, we get
$\cL_\sigma J=0$.
\smallskip

For the case $\n_\s(\h) = \q$, since we have
$\sigma\in <Jt>+\h$, it follows that
$\cL_{\sigma}J=0$.
\end{proof}
\medskip

\begin{Cr} 
We have $[\sigma,Jt]=0$;
in particular ${\rm Ad}(\exp Jt)_*\sigma=\sigma$. 
\end{Cr}

\begin{proof}
We have $(\cL_\sigma J)t=\cL_\sigma(Jt)-J\cL_\sigma t=0$
by Lemma 6. Since $[\sigma,t]=0$, it follows that $[\sigma,Jt]=0$.
\end{proof}
\medskip

\begin{Th} 
A compact homogeneous l.c.K. manifold $(M,h)$ is necessarily of Vaisman type;
that is, the Lee field $\xi$ is a Killing field with respect to any homogeneous l.c.K.
metric $h$ on $M$.
\end{Th}

\begin{proof}
We first consider the l.c.K. form $\bar{\Omega}, \bar{\psi}$ on $M$ averaged by the closure $K$ of
the $1$-parameter subgroup of $G$ generated by $Jt$.
We have $\bar\psi(\sigma)=\int_{K}{\rm Ad}(x)^* \psi(\sigma)=
\int_{K}\psi(\sigma)=1$ by Lemma 1 and
Corollary 4. Here we have normalized the volume of $K$ to $1$.
We also have $d \bar{\psi}(\sigma, Z)=0$ for any $Z \in \g$.
Hence we have $\bar{\psi}, \bar{\theta}=\theta, \sigma, t$ satisfying
the condition of Lemma 1; and thus by Lemma 2 we have
\[\bar \p=< t,\sigma> =<t, Jt>\]

Now we show that  $\cL_\xi \Omega=0,\, \cL_\xi J=0$ for the original l.c.K. form $\Omega$.
Since $Jt\in < t,\sigma>$ as shown above,
we have $\cL_{Jt}\Omega=0$ by Corollary 2.
As $\xi=-J\sigma \;({\rm mod}\, \h)$ from Corollary 1 and
$\sigma\in < t,Jt>$, we must have
$\xi \in < t,Jt> + \h$. Thus, $\cL_\xi \Omega=0$ and
$\cL_\xi J=0$.
Hence $\xi$ is a holomorphic Killing field with respect to $h$. 
\end{proof} 
\medskip

\section{Compact homogeneous l.c.K. manifolds of complex dimension~$2$}  

We know (due to Vaisman \cite{V1}, Gauduchon-Ornea \cite{GO} and Belgun \cite{Be})
that there is a class of Hopf surfaces
which admit homogeneous l.c.K. structures. We can show,
applying the above theorem, that the only compact homogeneous l.c.K.
manifolds of complex dimension $2$ are Hopf surfaces of homogeneous
type (see Theorem 3). We first determine, recalling a result of Sasaki (\cite{SS}),
all homogeneous complex structures on $G=S^1 \times SU(2)$,
or equivalently all complex structures on the Lie algebra
$\g={\mathfrak u}(2)$.
\medskip

\begin{Ps} 
Let $\g={\mathfrak u}(2)=\R \oplus \s {\mathfrak u}(2)$ be a reductive
Lie algebra with basis $\{T, X,Y,Z\}$ of $\g$, where $T$ is a generator
of the center $\R$ of $\g$, and
$$
X=\frac{1}{2}\left(
\begin{array}{cc}
\im & 0\\
0 & -\im
\end{array}
\right),
\;
Y=\frac{1}{2}\left(
\begin{array}{cc}
0 & \im\\
\im & 0
\end{array}
\right), 
\;
Z=\frac{1}{2}\left(
\begin{array}{cc}
0 & -1\\
1 & 0
\end{array}
\right)
$$
such that non-vanishing bracket multiplications are given by
$$[X,Y]=Z,\;[Y,Z]=X,\;[Z,X]=Y.$$
Then $\g$ admits a family of complex structures $J_\delta, \delta=c+\im d$
defined by
$$J_\delta(T-d X)=c X,\; J_\delta(c X)=-(T-d X),\; J_\delta Y= \pm Z,
\; J_\delta Z = \mp Y.$$
Conversely, the above family of complex structures exhaust all homogeneous
complex structures on $\g$.
\end{Ps}

\begin{proof} 
Let $\g_{\C}=\g {\mathfrak l} (2,\C)=\C+\s {\mathfrak l} (2,\C)$
be the complexficaion of $\g$,
which has a basis ${\mathfrak b}_{\C} =\{T,U,V,W\}$, where
$$U=\frac{1}{2}\left(
\begin{array}{cc}
-1 & 0\\
0 & 1
\end{array}
\right),
\;
V=\frac{1}{2}\left(
\begin{array}{cc}
0 & 0\\
1 & 0
\end{array}
\right),
\;
W=\frac{1}{2}\left(
\begin{array}{cc}
0 & 1\\
0 & 0
\end{array}
\right)
$$
with the bracket multiplication defined by
$$[U,V]=V,\;[U,W]=-W,\;[V,W]= \frac{1}{2} U.$$
Here we have
$$U=\im X,\;
V=\frac{1}{2}(Z-\im Y),\;
W=-\frac{1}{2}(Z+\im Y),$$
and their conjugations given by
$$\overline{T}=T,\; \overline{U}=-U,\; \overline{V}=-W,\;
\overline{W}=-V.
$$

We know that there is a one to one correspondence between complex
structures $J$ and complex subalgebras $\h$ such that
$\g_{\C} =\h + \overline{\h}$
and $\h \cap \overline{\h}=\{0\}$. 
Let $\ma$ be the subalgebra
of $\g_{\C}$ generated by $T$ and $\mb$ the subalgebra of
$\g_{\C}$ generated by $U,V,W$, then we have
$$\g_{\C} = \ma \oplus \mb$$
where $\ma=<T>_{\C}, \mb=<U, V, W>_{\C}$.
Let $\pi$ be the projection $\pi: \g_{\C} \rightarrow \mb$ and
$\mc$ the image of $\h$ by $\pi$, then we have 
$$\mb=\mc+\overline{\mc},$$
and ${\rm dim}\, \mc \cap \overline{\mc} =1$.
We can set a basis $\eta$ of $\h$ as
$\eta=\{P+Q, R\}\; (P \in \ma, Q,R \in \mb)$
such that $Q \in \mc \cap \overline{\mc}$ and
$\gamma =\{Q,R\}$ is a basis of $\mc$:
$$\h=<P+Q, R>_{\C},\; \mc=<Q, R>_{\C}.$$
Furthermore, we can assume that $Q+\overline{Q}=0$
so that $Q$ is of the form $aU+bV+\overline{b}W \,(a \in \R, b \in \C)$.
\smallskip

We first consider the case where $R=qV+rW \, (q,r \in \C)$. 
Since we have $[\g_{\C}, \g_{\C}]=\mb$, there is some $\alpha \in \C$
such that $[Q,R]=\alpha R$.
We see by simple calculation that if $b \not=0$,
then $q=sb, r=s \overline{b}$
for some non zero constant $s \in \C$. But then $\overline{R}=s R$,
contradicting to the fact that $\beta=\{Q, R,\overline{R}\}$ consists
a basis of $\mb$:
$$\mb=<Q, R, \overline{R}>_{\C}$$
Hence we have $b=0$, and $q \not=0, r =0$ with
$\alpha=a$ or $q =0, r \not=0$ with $\alpha=-a$.
Therefore we can take, as a basis of $\h$, $\eta=\{T+\delta U, V\}$ or
$\{T+\delta U, W\}$ with $\delta=c+\im d \in \C$:
$$\h=<T+\delta U, V>_{\C} \;{\rm or}\; <T+\delta U, W>_{\C}.$$
It should be noted that
the latter defines a conjugate complex structure of the former,
which are not equivalent but define biholomorphic
complex structures on its associated Lie group $G$.
\smallskip

In the case where $R=pU+qV+rW,\,p, q, r \in \C$ with $p \not=0$,
we show that there exists an automorphism $\widehat{\phi}$ on $\g_{\C}$ 
which maps $\h_0$ to $\h$, preserving the conjugation, where
$\h_{0}$ is a subalgebra of $\g_{\C}$ of the first type with $p=0$.
As in the first case, we must have $[Q,R]=\eta R$ for some non zero
constant $\eta \in \C$. We may assume that $p=1$.
We see, by simple calculation that
$b, q, r  \not= 0$ and
$$(a-\eta)q=b, (a+\eta)r=\overline{b},$$
from which we get
$$a^2+|b|^2=\eta^2 \;(\eta \in \R),$$
and
$$|q|^2-|r|^2 = \frac{4 a \eta}{|b|^2}.$$
Then an automorphism $\phi$ on $\mb$ defined by 
$$\phi(U)=\frac{1}{\eta}Q,\, \phi(V)=\frac{|b|}{2 \eta} R, \, \phi(W)=-\frac{|b|}{2 \eta} \overline{R},$$
extends to the automorphism $\widehat{\phi}$
on $\g_{\C}$ which satisfies the required condition.
\end{proof}
\medskip

\begin{Ps} 
Let $G=S^1 \times SU(2)$ (which is, as is well known,
diffeomorphic to $S^1 \times S^3$). Then
all homogeneous complex structures on
$G$ admit their compatible homogeneous l.c.K. structures,
defining a primary Hopf surfaces $S_{\lambda}$ which are compact
quotient spaces of the form
$W /\Gamma_{\lambda}$, where $W=\C^2 \backslash \{\0\}$
and $\Gamma_{\lambda}$
is a cyclic group of holomorphic automorphisms on $W$ generated by
a contraction $f: (z_1,z_2) \rightarrow
(\lambda z_1, \lambda z_2)$ with $ |\lambda| \not=0, 1$. Furthermore,
all of those l.c.K. structures are of Vaisman type.
\end{Ps}

\begin{proof} 
We consider the following canonical diffeomorphism $\Phi_\delta$,
which turns out to be biholomorphic for each homogeneous complex
structure $J_\delta$ on $\g$
and $\lambda_\delta$:
$$\Phi_\delta: \R \times SU(2) \longrightarrow W$$
defined by
$$(t,z_1,z_2) \longrightarrow (\lambda_\delta^t z_1, \lambda_\delta^t z_2),$$
where $SU(2)$ is identified with
$S^3=\{(z_1,z_2) \in \C \,|\;|z_1|^2+|z_2|^2=1\}$
by the correspondence:
$$
\left(
\begin{array}{cc}
z_1 & -\overline{z}_2\\
z_2 & \overline{z}_1
\end{array}
\right)
\longleftrightarrow (z_1,z_2),
$$
and $\lambda_\delta=e^{c+\im d}$.
Then we see that $\Phi_\delta$ is a biholomorphic map.
It is now clear that $\Phi_\delta$ induces a biholomorphism
between $G=S^1 \times SU(2)$ with homogeneous complex
structure $J_\delta$ and a primary Hopf surface 
$S_{\lambda_\delta}=W/\Gamma_{\lambda_\delta}$.
\smallskip

Let $t, x, y, z \in \g^*$ be the Maurer-Cartan forms corresponding to
$T, X,Y,Z \in \g$  in Proposition 2. Then we have
$$d z=- x \wedge y,\, d x=-y \wedge z,\, d y=- z \wedge x,$$
and
$$\Omega=- \theta \wedge \phi + d \phi,$$
where $\theta=t,\, \phi=\frac{1}{c} x,$
defines a l.c.K. form on $\g$ for the complex structure $J_{\delta}$ in 
Proposition 2.
Note that we have the Lee field $\xi=T-\frac{d}{c} \eta$, which
is irregular for an irrational $\frac{d}{c}$ while
the Reeb field $\eta=c X$, which is always regular. The Lee field $\xi$ is
a Killing field, since we have
$$h([\xi, U],V)+h(U,[\xi,V])= -d (h([X, U],V)+h(U,[X,V]))=0$$
for all $U,V \in \g$. Hence $(G; \Omega, J_\delta)$ is of Vaisman type.
\smallskip

A secondary Hopf surface with homogeneous l.c.K. structure can be obtained
as a quotient space of a primary Hopf surface $S_{\lambda_\delta}$ by
some finite subgroup of $G$.
For instance, $U(2)$ is a quotient Lie group of $G$ by
the central subgroup $\Z_2=\{(1, I), (-1, -I)\}$. In general
we have a secondary Hopf surface 
$G/\Z_m = S^1 \times _{\Z_m} SU(2),$
where ${\bf Z}_m$ is a
finite cyclic subgroup of $G$ generated by $c$:
\[
c= (\xi, \tau),\quad \tau = \left(
\begin{array}{@{}cc@{}}
\xi^{-1} & 0\\
0 & \xi
\end{array}
\right)\!,\enspace \xi^m = 1,
\]
with homogeneous l.c.K. structures induced from those on $G$ by the averaging
method (c.f. \cite{Has}). A (primary or secondary) Hopf surface defined as above
is called a {\em Hopf surface of homogeneous type}, which is a holomorphic principal
bundle over a $1$-dimensional projective space $\C P^1$ with fiber a $1$-dimensional
complex torus $T_{\C}^1$.
\end{proof}
\medskip

\begin{Th} 
Only compact homogeneous l.c.K. manifolds of complex dimension $2$
are Hopf surfaces of homogeneous type (up to biholomorphism).
\end{Th}

\begin{proof} 
It is sufficient to show that any compact homogeneous l.c.K.
manifold $M$ of complex dimension $2$ is a Hopf surface of homogeneous type as defined
in Proposition 3. As we have seen in Theorem~1, a compact homogeneous
l.c.K. manifold $M$ of complex dimension $2$ can be expressed as
$S^1 \times_{\Gamma} S$, where $S$ is a compact homogeneous contact
manifold of real dimension $3$ which admits a Hopf fibration over
$\C P^1$ with fiber $S^1$, and $\Gamma$ is a finite abelian group acting on the fiber $T_{\C}^1$
of the vibration $M \rightarrow \C P^1$.
These are exactly Hopf surfaces with homogeneous
l.c.K. structures as defined in Proposition 3.
Conversely a Hopf surface of homogeneous type admits a homogeneous l.c.K. structure as
defined in Proposition~3. 
\end{proof}

\section{Homogeneous l.c.K. structures on reductive Lie groups} 

A homogeneous l.c.K. structure on a Lie group $G$ is nothing but
a left invariant l.c.K. structure on $G$. Since $G$ can be expressed
as $\widehat{G}/\Delta$, where \textsf{}$\Delta$ is a finite subgroup of the
center of $\widehat{G}$, $G$ admits a l.c.K. structure $\Omega$
if and only if $\widehat{G}$ admits a l.c.K. structure $\hat{\Omega}$,
or equivalently
the Lie algebra ${\g}$ of $G$ admits a l.c.K. structure $\tilde{\Omega}$
in ${\wedge}\, {\g}^*$.
\medskip

\begin{Th} 
Let $\g$ be a reductive Lie algebra of dimension $2m$; 
that is, $\g = \mt + \s$,
where $\mt$ is an abelian and $\s$ a semi-simple Lie subalgebra of $\g$
with $\s=[\g,\g]$.
Then $\g$ admits a l.c.K. structure if and only if ${\rm dim}\, \mt = 1$
and ${\rm rank} \, \s =1$. In particular a compact Lie group admits
a homogeneous l.c.K. structure if and only if it is $U(2)$,
$S^1 \times SU(2) \cong S^1 \times Sp(1)$, or $S^1 \times SO(3)$;
and any homogeneous l.c.K. structure on a compact Lie group is of Vaisman type.
\end{Th}

\begin{proof}
Suppose that $\g$ admits a l.c.K. structure $\Omega$.
Since we have  $\h=\{0\}$, $\eta \in \s$ and thus ${\rm dim}\, \mt =1$.
If we apply the proof of Theorem 1 for the case $\h=\{0\}$,
we see that $\q =<\eta>=\{V \in \s |\, [\eta, V]=0\}$;
and thus ${\rm rank} \, \s =1$ (cf. \cite{BW}). We know all of the reductive Lie algebras
$\g=\mt + \s$ with ${\rm dim}\, \mt = 1$ and ${\rm rank} \, \s =1$:
$\R \oplus {\mathfrak sl}(2, \R)$ and 
${\mathfrak u}(2)=\R \oplus {\mathfrak su}(2)=\R \oplus {\mathfrak so}(3)$.
We show that all homogeneous l.c.K. structures on ${\mathfrak u}(2)$ are the ones 
we obtained in Proposition 3: $\Omega=-\theta \wedge \phi+ d \phi$; and
they are all of Vaisman type. In fact, any l.c.K. form $\Omega'$ is of the form
$$\Omega'=-\theta \wedge \psi + d \psi,$$
where we can set $\theta=t$ and $\psi=a x+b y+c z \,(a, b, c \in \R)$; and thus
$d \psi=-(a \,y \wedge z+ b \,z \wedge x + c \,x \wedge y)$.
For the complex structure $J_{\delta}$ in Proposition~2, we denote by $A$ the
$4\times4$-matrix determined by $h'(U,V)=\Omega'(J_{\delta} U,V)$
for $U,V=T,X,Y,Z$. By the condition that $A$ is a positive-definite symmetric matrix,
we can see by calculation that $b=c=0$; and thus $A=a I_4$. Hence
$\Omega'$ is equal to the original $\Omega$ up to constant multiplication.
\end{proof}
\medskip

\begin{Ex} 
{\rm We can also consider $=S^1 \times S^3$ as a compact homogeneous
space $\tilde{G}/H$, where $\tilde{G}=S^1 \times U(2)$
with its Lie algebra 
$\tilde{\g}=\R \oplus {\mathfrak u}(2)$ and
$H=U(1)$ with its Lie algebra $\h$.
Then, we have a decomposition $\tilde{\g}=\m+\h$ for the subspace $\m$ of
$\tilde{\g}$ generated by $S, T, Y, Z$ and $\h$ generated by $W$, where
$$ 
S=\frac{1}{2}\left(
\begin{array}{cc}
\im & 0\\
0 & \im
\end{array}
\right),
\;
W=\frac{1}{2}\left(
\begin{array}{cc}
0 & 0\\
0 & \im
\end{array}
\right).
$$
Since we have $S=X+2 W$, we can take $\m'$ generated by
$T,X,Y,Z$ for $\m$; and homogeneous l.c.K. structures
on $\tilde{G}/H$ are the same as those on $G$. In other words
any homogeneous l.c.K. structures on $G$ can be extended as
those on $\tilde{G}/H$. 
\smallskip

Furthermore, we can construct locally homogeneous
l.c.K. manifolds $\Gamma \backslash \hat{G}/H$ for some discrete
subgroups $\Gamma$ of $\hat{G}$, where $\hat{G}=\R \times U(2)$.
For instance, let $\Gamma_{p,q}\, (p,q \not=0)$ be a discrete
subgroup of $\hat{G}$:
$$ 
\Gamma_{p,q}=\{(k,\left(
\begin{array}{cc}
e^{\im p k} & 0\\
0 & e^{\im q  k}
\end{array}
\right))
\in \R \times U(2)\, |\; k \in \Z
\}.
$$
Then $\Gamma_{p,q} \backslash \hat{G}/H$ is biholomorphic to
a Hopf surface $S_{p,q} =W/\Gamma_{\lambda_1,\lambda_2}$, where
$\Gamma_{\lambda_1,\lambda_2}$ is the cyclic group of automorphisms on
$W$ generated by

$$\phi: (z_1,z_2) \longrightarrow (\lambda_1 z_1, \lambda_2 z_2)$$
with $\lambda_1=e^{r+\im p}, \lambda_2=e^{r+\im q}, r \not=0.$
In fact, if we take a homogeneous complex structure $J_r$ on $\hat{G}/H$
induced from the diffeomorphism $\Phi_r: \hat{G}/H \rightarrow W$ defined by
$(t, z_1,z_2) \longrightarrow (e^{r t}  z_1, e^{r t} z_2)$, $\Phi_r$ induces
a biholomorphism between $\Gamma_{p,q} \backslash \hat{G}/H$ and $S_{p,q}$.
Note that in case $p=q$, $S_{p,q}$ is biholomorphic to $S_{\lambda}$ with
$\lambda=r+\im q$.
}
\end{Ex}
\medskip

We have an example of a compact locally homogeneous l.c.K. manifold of non-compact reductive Lie group
which is not of Vaisman type (\cite{ACHK}).

\begin{Ex} 
{\rm There exists a homogeneous l.c.K. structure on $\g=\R \oplus {\mathfrak sl}(2, \R)$,
which is not of Vaisman type. We can take a basis $\{W, X, Y, Z\}$ for $\g$ with
bracket multiplication defined by
$$[X,Y]=-Z,\, [Z,X]=Y,\, [Z,Y]=-X,$$
and all other brackets vanish.
We have a homogeneous complex
structure defined by
$$J Y=X, J X=-Y, J W=Z, J Z=-W,$$
and its compatible l.c.K. form $\Omega$ on $\g$ defined by
$$\Omega = z \wedge w + x \wedge y,$$
with the Lee form $\theta=w$, where $x, y, z, w$ are the Maurer-Cartan
forms corresponding to $X,Y,Z,W$ respectively. 
We can take another
l.c.K. form $\Omega_{\psi} = \psi \wedge w+ d \psi$, where 
$\psi=b y + c z \,(b, c \in \R)$ with $b > c >0$ and $b^2-c^2=b$,
making the corresponding metric $h_{\psi}$ positive definite. The Lee field $\xi$ is
given as $\xi=\frac{1}{b^2-c^2} (bW+cX)$. It is easy to check that
$h([\xi, X],Y) + h(X, [\xi,Y]) \not= 0;$
and thus $\xi$ is not a Killing field.
\smallskip

For any lattice $\Gamma$ of $G = \R\times SL(2)$ with the above homogeneous l.c.K. structure,
we get a complex surface $\Gamma \backslash G$ (properly elliptic surface)
with locally homogeneous non-Vaisman l.c.K. structure.}
\end{Ex}
\medskip

\noindent{\bfseries Acknowledgement.}
The authors would like to express their special thanks to {D. Alekseevsky} and 
V. Cort\'es for valuable discussions and useful comments. The authors also
would like to thank D.~Guan for beneficial discussions by e-mail; and
A. Moroianu for important remarks.


\bigskip

\address{Department of Mathematics\\
Faculty of Education\\
Niigata University\\
Niigata 950-2181\\
JAPAN}
{hasegawa@ed.niigata-u.ac.jp}
\address{
Department of Mathematics\\
Tokyo Metropolitan University\\
Hachioji 192-0397\\
JAPAN}
{kami@tmu.ac.jp}

\end{document}